\documentclass[11pt]{amsart}
\usepackage{a4wide}
\usepackage{amsmath,amssymb,amsfonts,amsthm}
\usepackage{mathtools}
\usepackage{xcolor}
\usepackage{xargs}
\usepackage[normalem]{ulem}
\usepackage{mathrsfs}
\usepackage{esint}
\usepackage[colorlinks,citecolor=blue,urlcolor=blue,bookmarks=false,hypertexnames=true]{hyperref} 
\theoremstyle{plain}
\usepackage[utf8]{inputenc}

\usepackage{graphicx}

\newtheorem{theorem}{Theorem}[section]
\newtheorem{lemma}[theorem]{Lemma}
\newtheorem{proposition}[theorem]{Proposition}

\theoremstyle{definition}
\newtheorem{remark}[theorem]{Remark}
\newtheorem{definition}[theorem]{Definition}

\numberwithin{equation}{section}

\newcommand{\bd}{\partial}
\newcommand{\R}{\mathbb{R}}

\newcommand{\N}{\mathbb{N}}

\newcommand{\T}{\mathbb{T}}

\newcommand{\ud}{\,\textnormal{d}}

\let\div\undefined
\newcommand{\div}{\textnormal{div}}
\newcommand{\udH}{\ud\mathcal H^{N-1}}
\newcommand{\dist}{\textnormal{dist}}
\newcommand{\sd}{\textnormal{sd}}

\title[Stability Modified Mullins-Sekerka]{Dynamical stability for the periodic modified \\Mullins-Sekerka flow}

\author[Daniele De Gennaro]
 {Daniele De Gennaro}
\address[Daniele De Gennaro]{Dipartimento di Matematica, Università degli Studi di Milano, Via Saldini 50, I-20133 Milano, Italy}
\email[D. De Gennaro]{daniele.degennaro@unimi.it}

 \author[Anna Kubin]
 {Anna Kubin}
 \address[Anna Kubin]{Institute of Analysis and Scientific Computing, Technische Universit\"at Wien, Wiedner Haupstrasse 8-10, 1040 Vienna, Austria}
 \email[Anna Kubin]{anna.kubin@asc.tuwien.ac.at}

\date{}
\begin{document}

\begin{abstract}
    We prove dynamical stability in arbitrary dimension for the modified Mullins--Sekerka flow, the gradient flow of the sharp-interface Ohta--Kawasaki energy on the flat torus. Specifically, we show that if  an initial set has the same volume as a strictly stable critical set for the energy and is sufficiently close to it in $C^{3,\alpha}$, then the flow exists for all times and converges exponentially fast, in every $C^k$ norm, to a translate of that critical set. The proof relies on a quantitative Alexandrov-type estimate for strictly stable critical sets of the energy.
\end{abstract}

\maketitle

\section{Introduction}
The Ohta--Kawasaki energy was introduced in~\cite{OhtKaw} to describe
microphase separation in diblock copolymer melts. Given a set $E\subset\mathbb{T}^N$, the sharp-interface Ohta--Kawasaki energy is defined as
\begin{equation}\label{intro:OK-energy}
	\mathcal J(E)
	:=
	P(E)
	+
	\gamma
	\int_{\T^N}\int_{\T^N}
	G(x,y)
	\bigl(u_E(x)-m\bigr)
	\bigl(u_E(y)-m\bigr)
	\ud x\ud y ,
\end{equation}
where $\gamma\geq0$, $u_E:=2\chi_E-1$, $m=\fint_{\T^N}u_E$ and
$G$ denotes the Green function of the Laplacian in $\T^N$.
The two terms in~\eqref{intro:OK-energy} model competing effects: the perimeter
penalizes the creation of interfaces, while the nonlocal term accounts for long-range repulsive interactions. Their competition gives rise to the
formation of nontrivial periodic patterns. 
Critical points and local minimizers of the
Ohta--Kawasaki functional  have been extensively studied. In particular,  explicit  equilibria and their
stability properties were analyzed in~\cite{RenWei1, RenWei2},
the   relation between strict stability and local minimality
was studied in~\cite{ChoSte,AFM}, and periodic locally minimizing critical
points were constructed in~\cite{Cri}. We also refer to~\cite{AlaBroRusTop, ChoPel, GolMurSer}
for results on the structure and distribution of minimizers in
the small-volume and droplet regimes. We emphasize, however, that there is still no characterization of critical points under general assumptions.

A natural evolution associated with~\eqref{intro:OK-energy} is the modified Mullins--Sekerka flow, also known as the nonlocal Mullins--Sekerka flow. More precisely, a smooth family of sets $(E_t)_{t\ge 0}$ is said to evolve by the modified Mullins--Sekerka flow if
\begin{equation}\label{intro:MS}
	\begin{cases}
		V_t=[\partial_{\nu_t}w_t]
		&\text{on }\partial E_t,\\
		\Delta w_t=0
		&\text{in }\T^N\setminus\partial E_t,\\
		w_t=H_{E_t}+4\gamma v_{E_t}
		&\text{on }\partial E_t,\\
		-\Delta v_{E_t}=u_{E_t}-m
		&\text{in }\T^N,
	\end{cases}
\end{equation}
with $v_E$ having zero average.
Here $V_t$ denotes the outer normal velocity of $\partial E_t$ and
$[\partial_{\nu_t}w_t]$ the jump of the normal derivative across the
interface. 
When $\gamma=0$, system~\eqref{intro:MS} reduces to the classical
Mullins--Sekerka flow~\cite{MulSek}.
This flow arises as the sharp-interface  limit of the evolution associated with
the diffuse Ohta--Kawasaki energy. 
The derivation of this flow as a sharp-interface limit of the Ohta–Kawasaki phase-field equation via a singular perturbation argument was first established in~\cite{NisOhn}, and the rigorous convergence of the phase-field model to the sharp-interface evolution   was proved in~\cite{Le}. Short-time
well-posedness for the classical Mullins--Sekerka problem was established in~\cite{XinJiaFah, ESSIM98}, while the corresponding modified problem was treated in~\cite{EscNis}.

Like the classical Mullins--Sekerka flow, system~\eqref{intro:MS} preserves
the volume of the evolving sets and dissipates the energy ${\mathcal J}$. More precisely,
every smooth solution of~\eqref{intro:MS} satisfies
\begin{equation}\label{intro:dissipation}
	\frac{\ud}{\ud t}|E_t|=0,
	\quad
	\frac{\ud}{\ud t}{\mathcal J}(E_t)
	=
	-\int_{\T^N}|\nabla w_t|^2 \ud x .
\end{equation}
Moreover, this system can be interpreted as the gradient flow of
${\mathcal J}$ with respect to a suitable metric on the space of interfaces.  This gradient-flow interpretation has been developed in the smooth setting in~\cite{NieOtt, Le} and, more recently, at the level of weak solutions in ~\cite{HenSti}.

In view of this gradient-flow structure, it is natural to expect  that  static
stability of a critical point should imply  dynamical stability. This question
was addressed in~\cite{AFJM} in dimensions two and three. More precisely, the authors proved that a solution of the modified Mullins--Sekerka flow starting sufficiently close to a strictly stable critical set exists for all time and converges exponentially fast, in a suitable Sobolev norm, to a translate of the reference set. Their argument is variational and is
based on differentiating the dissipation~\eqref{intro:dissipation},
exploiting the positivity of the second variation and controlling the
nonlinear remainder terms. 
We also refer to~\cite{DelDiaMan} for a detailed account of this
result.

The arguments of~\cite{AFJM}, however, rely on 
interpolation estimates depending crucially on the dimension. 
A different approach was introduced in~\cite{DeKu} for the discrete volume-preserving mean curvature flow on the flat torus, extending the ideas developed in~\cite{MorPonSpa} in the Euclidean setting. This approach was subsequently employed in~\cite{DegDiaKubKub} to prove, in arbitrary dimension, the dynamical stability of strictly stable sets for both the volume-preserving mean curvature flow and the surface diffusion flow. The key ingredient in these results is a quantitative Alexandrov-type estimate.

Quantitative versions of Alexandrov's theorem  were first obtained  in~\cite{KM, JulNii, MorPonSpa}. 
These results play an important role in the analysis of geometric evolutions, as they provide the correct estimates to show compactness for curvature-driven gradient flows, including the mean curvature flow, surface diffusion flow and Mullins--Sekerka flow.  
In the context of weak solutions to geometric flows, such  estimates, coupled with a characterization of the critical points of the energy, 
 can be used to prove the asymptotic convergence of weak solutions towards critical points under low regularity  of the initial data. For more recent results regarding Alexandrov-type estimates, we refer to~\cite{JulMorPonSpa, JulMorOroSpa, KimKwon, AryDegKub}.

The purpose of the present paper is to develop this strategy for the
 modified Mullins--Sekerka flow.  
 In
particular, we prove the  stability of the flow in arbitrary dimension, thereby removing the dimensional restriction in~\cite{AFJM}. Moreover, our result  yields exponential convergence to a translate of the reference set in every $C^k$-norm, whereas the result of~\cite{AFJM} establishes convergence only in a suitable Sobolev norm.
Our main result is the following.

\begin{theorem}\label{thm:main}
	Let $E\subset\T^N$ be a strictly stable critical set for ${\mathcal J}$, and let
	$\alpha\in(0,1)$. There exist $\delta>0$ and constants $c>0$,
	$C_k>0$, $k\in\N$, with the following property.
	
	If $E_0\subset\T^N$ is a $C^{3,\alpha}$-set satisfying
	\[
	|E_0|=|E|,
	\qquad
	\dist_{C^{3,\alpha}}(E_0,E)\leq\delta,
	\]
	then the modified Mullins--Sekerka flow starting from $E_0$ admits a unique
	smooth solution $(E_t)_{t\geq0}$ for all positive times. Moreover, there
	exists $\sigma_\infty\in\T^N$ such that 
	\[
	\dist_{C^k}(E_t,E+\sigma_\infty)
	\leq
	C_k e^{-ct},\qquad \text{for all }k\in\N, t\ge 1.
	\]
\end{theorem}

We now outline the main steps of the proof.
We first prove a local quantitative Alexandrov estimate for the Ohta--Kawasaki energy. More precisely, let $E$ be a strictly stable critical set and suppose that $F$ can be represented as a normal graph over   $E$ with  height function $f$ sufficiently small. We show the following inequality  
\begin{equation}\label{intro:Alex}
	\|f\|_{H^1(\bd E) }
	\leq
	C
	\left\|
        H_F+4\gamma v_F
        -
        \overline{H_F+4\gamma v_F}
    \right\|_{L^2(\partial F)} .
\end{equation}
In particular, if $(E_t)_{t\in[0,T^*)}$ denotes the modified Mullins--Sekerka flow starting from an initial set $E_0$ sufficiently close to $E$, then this estimate, combined with the dissipation identity~\eqref{intro:dissipation}, yields  exponential decay of the energy gap ${\mathcal J}(E_t)-{\mathcal J}(E)$ up to the extinction time $T^*$.

The proof of~\eqref{intro:Alex} follows the strategy of~\cite{DeKu}. We compare two expansions of the first variation of ${\mathcal J}$ around
the reference set $E$ and then exploit the coercivity of the second variation $\delta^2 {\mathcal J}(E)$. 
A related approach has recently been developed in~\cite{AlbCozMasMir} for the fractional perimeter energy, where the authors additionally require to prove a  uniform  coercivity estimate for $\delta^2 {\mathcal J}(F)$ holding for sets $F$ in a  neighborhood of the stable set $E$. 

Finally, we prove that the extinction time is infinite  and that the flow converges to a translate of $E$ in every $C^k$-norm, for all $k\in\N$.
For this purpose, we establish short-time regularity estimates that are uniform for initial data with bounded $C^{3,\alpha}$-norm. This is done  combining the approaches of~\cite{EscNis} and
\cite{XinJiaFah}. After a suitable change of unknown that transforms the moving-boundary problem into a problem on a fixed domain, the nonlocal 
contribution appears as a lower-order inhomogeneous term in the elliptic
problem associated with the Mullins--Sekerka system.
The results of~\cite{XinJiaFah} can thus be adapted
to obtain a uniform lower bound for the existence time, together with parabolic Schauder estimates. The crucial point is that the estimates depend only on an upper bound for the
$C^{3,\alpha}$-norm of the initial graph.

These estimates together with the exponential decay of the energy allows us
to restart the flow on successive intervals of uniform length and thereby obtain global existence.
The energy decay, combined with interpolation, then yields  the exponential convergence
in every $C^k$-norm, up to translations. Finally, we can show
that also the translations converge exponentially fast, yielding the
convergence of the whole flow to a single translate of $E$.


\section{Preliminaries}
In this section we collect some preliminary results used throughout the paper.

We consider the sharp-interface Ohta--Kawasaki energy
\begin{equation}\label{energy}
{\mathcal J}(E):=P(E)+F(E):= P(E)+ \gamma \int_{\T^N}\int_{\T^N} G(x,y) (u_E(x)-m)(u_E(y)-m) \ud x \ud y,
\end{equation}
where $u_E=\chi_E-\chi_{E^c}=2\chi_E-1$ and $m=\fint_{\T^N} u_E$. We recall that $G$ is the Green's function of the Laplacian in the flat torus, i.e., for every $x \in \T^N$, it is the unique solution to
\begin{equation}\label{eq:fund-T}
\begin{cases}
    -\Delta_y G(x,y)= \delta_x(y)-1,\quad y \in \T^N\\
    \int_{\T^N} G(x,y) \ud y=0, \quad y \in \T^N.
\end{cases}
\end{equation}
Note that $G(x,y)=G(y,x)$. Indeed, by integration by parts, $0=\int G(x,\cdot)\Delta G(y,\cdot)-G(y,\cdot)\Delta G(x,\cdot)=G(x,y)-G(y,x)$. Moreover, the translation invariance of the Laplacian implies that $G(x,y)=G(x-y,0)$. 

The potential $G$ admits the decomposition 
\[
G(x,y)=h(x-y) + r(x-y),
\]
where $h$ satisfies   $h(\rho)=C_N \rho^{2-N}$ for $N\ge 3$ and $h(\rho)=-\log(|\rho|)/2\pi$ for $N=2$ in a neighborhood of zero, while $r$ is smooth and one-periodic.
This can be proved by looking for a solution to~\eqref{eq:fund-T} of the form  $\eta h+r,$ where $\eta$ is a periodic cut-off function with compact support. 
In particular, $h$ satisfies the following H\"older estimate
\begin{equation}\label{eq:holderh}
|h(\rho_1)-h(\rho_2)|\le C\max\Big\{\frac{1}{|\rho_1|^{N-2+\beta}},\frac{1}{|\rho_2|^{N-2+\beta}} \Big\}|\rho_1-\rho_2|^\beta
\end{equation}
for every $\rho_1$, $\rho_2 \in \T^N \setminus \{0\}$,
where $C$ is a positive constant independent of $\rho_1$ and $\rho_2$.

We define 
\begin{equation}\label{def vE}
    v_E(x):=\int_{\T^N} G(x,y)(u_E(y)-m) \ud y,
\end{equation}
which satisfies
\[
-\Delta v_E=u_E-m \quad \text{in} \quad \T^N, \quad \int_{\T^N} v_E \ud x=0.
\]
Observe that, since $G$ has zero mean, we have  $v_E(x)=2 \int_E G(x,y) \ud y$ and
\begin{equation*}
{\mathcal J}(E)= P(E)+ 4\gamma \int_{E}\int_{E} G(x,y) \ud x \ud y.
\end{equation*}

We recall the following estimates.

\begin{lemma}\label{lem:est-G}
    Let $E\subset \T^N$ be a $C^{2}$ set. Then, it holds 
    \begin{equation}\label{eq:est-G}
        \sup_{x\in\partial E}\int_{\bd E} |G(x,y)|\udH_y <+\infty, \quad \| v_E\|_{C^{1,\beta}(\bd E)} <+\infty \text{ for all }\beta\in (0,1).
    \end{equation}
\end{lemma}
The first bound  can be shown using the  integrability of $G$ over $\bd E$, the second one follows from elliptic regularity estimates since $u_E-m\in L^\infty(\T^N)$.

\subsection{Stable sets}
We recall the expressions for the first and second variations of ${\mathcal J}$ and the notion of stable sets used in the rest of the paper. 

Let $E\subset \T^N$ be a set of class $C^2$ and $X:\T^N\to \R^N$ be a vector field of class $C^2$. Consider the associated flow $\Phi:\T^N\times (-1,1)\to\T^N$ defined by $\frac {\bd \Phi}{\bd t}=X(\Phi),\ \Phi(\cdot,0)=Id$. We define the \textit{first and second variation of ${\mathcal J}$ at $E$} with respect to $\Phi$, respectively, by
\[
\delta {\mathcal J}(E)[X]:=\dfrac {\ud} {\ud t}\Big\lvert_{t=0}{\mathcal J}(E(t)), \quad \delta^2 {\mathcal J}(E)[X]:=\dfrac {\ud^2} {\ud t^2}\Big\lvert_{t=0}{\mathcal J}(E(t)), 
\]
where $E(t)=\Phi(\cdot,t)(E).$ 
Given a set $E$ of class $C^2$, we use the notation $\nu_E, H_E, B_E$ to denote the outer normal, scalar mean curvature and second fundamental form of $E$, respectively.  Also, given a vector field $X:\T^N\to \R^N$  we denote by  $X_\tau$ its tangential component along $\bd E$, defined as $X-(X\cdot\nu_E)\nu_E$. Similarly, $\div_\tau$ denotes the tangential divergence on $\bd E$.

Given a function $\varphi\in C^2_c(\T^N)$ we will use the notation 
\[
\delta {\mathcal J}(E)[\varphi]:= \delta {\mathcal J}(E)[\varphi\nu_E ],
\]
where $\nu_E$ denotes an extension of $\nu_E$ to a neighborhood of $\bd E$.

\begin{theorem}{\cite[Theorem~3.1]{AFM}}
	If $E$ and $X$ are as above, we have
 \begin{equation}\label{varprima}
 \delta {\mathcal J}(E)[X]=\int_{\bd^* E} (H_E+4\gamma v_E) \nu_E\cdot X\udH, 
 \end{equation}
and
\begin{align*}
	\delta^2 {\mathcal J}(E)[X]=&\int_{\bd E} \left( |D_\tau (X\cdot \nu_E)|^2-|B_E|^2 (X\cdot \nu_E)^2 \right)\udH \\
    &+8\gamma\int_{\bd E}\int_{\bd E}G(x,y)(X\cdot \nu_E)(x) (X\cdot \nu_E)(y) \udH(x) \udH(y)\\
    &+4\gamma \int_{\bd E}\partial_{\nu_E}v_E (X\cdot \nu_E)^2\udH-\int_{\bd E} (4\gamma v_E+H_E) \div_\tau (X_\tau(X\cdot \nu_E)) \udH\\
	&+\int_{\bd E} (4\gamma v_E+H_E)(\div X)(X\cdot \nu_E)  \udH.
\end{align*}
\end{theorem}

\begin{definition}
	Let $E$ be a set of class $C^1$.
	Given a measurable function $f:\bd E \to \R$ such that $\|f\|_{L^{\infty}(\bd E)}$ is sufficiently small, we set
	\begin{equation}
		\label{normal_deformation}
		\bd E_f:=\{ x+f(x)\nu_E(x) : x\in \bd E \},
	\end{equation}
	and we call $E_f$ the \textit{normal deformation} of $E$ with height function $f$.
\end{definition}

We recall that, due to the translation invariance of the  functional ${\mathcal J}$, the second variation degenerates along flows of the form $\Phi(x,t)=x+t\eta,$ where 
$\eta\in \R^N$. In view of this, it is convenient to introduce the subspace $T(\bd E)$ of $\tilde H^1(\bd E):=\left\lbrace \varphi\in H^1(\bd E) : \int_{\bd E} \varphi \udH=0   \right\rbrace$ generated by the functions $\nu_i=\nu \cdot e_i$ for $i=1,\dots, N$. Its $L^2$-orthogonal subspace will be denoted by $T^\perp(\bd E)$ and is given by
\[ T^\perp(\bd E)= \Big\{ \varphi\in\tilde H^1(\bd E) : \int_{\bd E} \varphi  \nu_i\udH=0,\ i=1,\dots,N \Big\}.\]

We are now in a position to define the class of sets to which our main result applies.

\begin{definition}
	\label{definizione_strictly_stable_set}
We say that a smooth set $E \subset \T^N$  is \textit{a stable set} (for the energy ${\mathcal J}$) if it is a critical set for ${\mathcal J}$, i.e., 
\[\delta {\mathcal J}(E)[\varphi]=0, \qquad \forall \varphi \in \tilde H^1(\bd E),\] 
and has strictly positive second variation, i.e.,
	\[\delta^2 {\mathcal J}(E)[\varphi]>0,\qquad \forall \varphi \in T^{\perp} (\bd E) \setminus \{0\}.\]
\end{definition}
We observe that, if $E$ is a stable set and we consider volume-preserving variations, then the second variation of ${\mathcal J}$ at $E$ reduces to (see~\cite{AFM})
\begin{align*}
	\delta^2 {\mathcal J}(E)[X]=&\int_{\bd E} \left( |D_\tau (X\cdot \nu_E)|^2-|B_E|^2 (X\cdot \nu_E)^2 \right)\udH \\
    &+8\gamma\int_{\bd E}\int_{\bd E}G(x,y)(X\cdot \nu_E)(x) (X\cdot \nu_E)(y) \udH(x) \udH(y)\\
    &+4\gamma \int_{\bd E}\partial_{\nu}v_E (X\cdot \nu_E)^2\udH.
\end{align*}

We recall the following coercivity estimate. 
\begin{lemma}{\cite[Lemma 3.6]{AFM}}\label{lemma coercive}
    Assume that $E$ is a stable set for ${\mathcal J}$. Then there exists $m_0>0$ such that 
    \begin{equation*}
        \delta^2 \mathcal{J}(E)[\varphi]\ge m_0\|\varphi\|_{H^1(\bd E)}^2\qquad \text{for all } \varphi\in T^\perp(\bd E).
    \end{equation*} 
\end{lemma}
As a corollary, we deduce a coercivity estimate for functions which are approximately orthogonal to $T(\partial E)$. The proof follows the same argument as that of~\cite[Lemma 2.9]{DeKu}. It is based on the continuity in $H^1(\partial E)\times H^1(\partial E)$ of the bilinear form associated with  $\delta^2 {\mathcal J}(E)$ and it is thus omitted.
\begin{lemma}\label{lem:coerc}
    Assume that $E$ is a stable set for ${\mathcal J}$. Then, there exist $\delta,C>0$ such that if  $ \varphi\in H^1(\bd E)\cap C^1(\bd E)$ satisfies  $\|\varphi\|_{C^1(\bd E)}\le \delta$ and  
    \[
    \Big|\int_{\bd E}\varphi \udH\Big| +\Big|\int_{\bd E}\varphi \nu_E \udH\Big|\le \delta  \|\varphi\|_{L^2(\bd E)},
    \]
    then it holds 
    \begin{equation}\label{coercive}
        \delta^2 \mathcal{J}(E)[\varphi]\ge \frac{m_0}2\|\varphi\|_{H^1(\bd E)}^2,
    \end{equation}
    where  $m_0$ is defined in  Lemma~\ref{lemma coercive}.
\end{lemma}

\section{Quantitative Alexandrov Theorem}\label{sec:contazzi}

The aim of this section is to prove the following stability estimate.

\begin{theorem}\label{teo alex}
Let $E \subset \T^N$ be a stable set.
There exist $\delta^* \in(0,1/2)$ and $C>0$, depending only on $E$, with the following property: for any $f\in C^1(\bd E)\cap H^2(\bd E)$ such that $\|f\|_{C^1(\bd E)}\leq \delta^*$ and satisfying
\begin{equation}
    \left |\int_{\bd E} f\ud \mathcal{H}^{N-1}\right |\le \delta^* \|f\|_{L^2(\bd E)},\qquad \left |\int_{\bd E} f\nu_E\ud \mathcal H^{N-1}\right|\leq \delta^*  \|f\|_{L^2(\bd E)},
    \label{hp alex}
\end{equation}
we have 
\begin{equation}
    \|f\|_{H^1(\bd E)}\leq C\|H_{E_f}(\cdot+f(\cdot)\nu_E(\cdot))-\widetilde H_{E_f}+ 4\gamma (v_{E_f}(\cdot+f(\cdot)\nu_E(\cdot))-\widetilde v_{E_f})\|_{L^2(\bd E)},
    \label{e:aleq_gen}
\end{equation}
where we have set
\[
\widetilde H_{E_f} := \fint_{\bd E} H_{E_f}(x+f(x)\nu_E(x)) \udH(x),\quad \widetilde v_{E_f}:= \fint_{\bd E} v_{E_f}(x+f(x)\nu_E(x)) \udH(x).
\]
\end{theorem}

\begin{remark}
    We remark that the conditions \eqref{hp alex} are related to geometric conditions satisfied by the set $E_f$. The first assumption is simply the first order condition ensuring that $|E|=|E_f|$. In some special geometries, the second condition is satisfied if $E_f$ has the same barycenter as $E$. We refer to \cite{DeKu} for further details.
\end{remark}

\begin{proof}
For ease of notation, we will often omit the integration measure when it is clear from context.

As a first step we develop the first variation of the energy ${\mathcal J}=P+\gamma F$. 
We note that the perimeter term can be treated as in the proof of~\cite[Theorem 1.3]{DeKu}, hence we only focus on the nonlocal term.
Moreover, without loss of generality, we can assume that the right-hand side of~\eqref{e:aleq_gen} is less than or equal to one, otherwise the estimate is trivial.

Consider the map $(x,s) \mapsto x+s \nu_E(x)$  defined in a tubular neighborhood of $\partial E$ and denote by $J(x,s)$ its Jacobian. Recall the Taylor expansion $J(x,s)=1+H_E(x)s+O(s^2)$.
We also consider $\Phi_f(x):=x+f(x) \nu_E(x)$ for $x \in \partial E$ and let $\varphi \in H^1(\partial E)$.
By the first variation formula~\eqref{varprima} and the area formula,  we have
\[
\delta F(E_f)[\varphi] = \dfrac {\ud} {\ud \varepsilon}\Big\lvert_{\varepsilon=0}  F(E_{f+\varepsilon \varphi})
=
4\gamma\int_{\partial E}
v_{E_f}(\Phi_f(x))\varphi(x)J(x,f(x))\udH_x.
\]
We now expand the difference $v_{E_f}\circ\Phi_f-v_E$. Recalling that
$v_E(x)=2\int_E G(x,y)\ud y$ and  using the normal parametrization coordinates, we infer 
\begin{equation}\label{eq:exp-v}
v_{E_f}(\Phi_f(x)) = v_E(\Phi_f(x))+2\int_{\partial E}\int_0^{f(y)}
G(\Phi_f(x),y+s\nu_E(y))J(y,s)\ud s\udH_y .
\end{equation}
Since $v_E\in C^{1,\beta}(\mathbb T^N)$, we deduce
\[
v_E\circ \Phi_f-v_E
=
\partial_{\nu_E}v_E\, f
+
O(|f|^{1+\beta}).
\]
We next consider the second term in~\eqref{eq:exp-v}. Using the decomposition $G=h+r$  and the H\"older estimate~\eqref{eq:holderh}, for
$|s|\leq |f(y)|$ we have
\begin{equation*}
\left|
h\bigl(\Phi_f(x)-y-s\nu_E(y)\bigr)-h(x-y)
\right|\leq
C
\frac{
|\Phi_f(x)-x-s\nu_E(y)|^\beta
}{
|x-y|^{N-2+\beta}
}
\leq
C\frac{\|f\|_{C^0(\partial E)}^\beta}
{|x-y|^{N-2+\beta}},
\end{equation*}
where the denominators in \eqref{eq:holderh} are estimated using that the projection is Lipschitz in a fixed tubular neighborhood of $\partial E$
\[
|x-y|
\leq
C\left|
\Phi_f(x)-y-s\nu_E(y)
\right|.
\]
Moreover, $r$ satisfies the above estimate with
$\|f\|_{C^0(\partial E)}$ in place of $\|f\|_{C^0(\partial E)}^\beta$ as it is smooth. 
Since $N-2+\beta<N-1$, 
the kernel $|x-y|^{-(N-2+\beta)}$ is integrable on $\partial E$, uniformly
with respect to $x$. 
Therefore, 
using Cauchy-Schwartz, 
we deduce 
\begin{align*}
&\Big|
\int_{\partial E}\int_{\partial E} \int_0^{f(y)}
\big( G(\Phi_f(x),y+s\nu_E(y))
-
 G(x,y)\big) J(y,s) \varphi(x)  \ud s\udH_y \udH_x
\Big|
\\
&\le C\|f\|_{C^0(\partial E)}^\beta \int_{\bd E}\int_{\bd E}  \frac{(|f(y)|+O(|f(y)|^2))|\varphi(x)| }{|x-y|^{N-2+\beta}}\udH_y \udH_x \\
&\le
 C\|f\|_{C^0(\partial E)}^\beta  \left( \int_{\bd E}\int_{\bd E}  \frac{|f(y)|^2}{|x-y|^{N-2+\beta}}\right)^\frac12\left( \int_{\bd E}\int_{\bd E}  \frac{|\varphi(x)|^2}{|x-y|^{N-2+\beta}}\right)^\frac12 
\\
&\le
C\|f\|_{C^0(\partial E)}^\beta
\|f\|_{L^2(\partial E)}\|\varphi\|_{L^2(\bd E)}.
\end{align*}
Finally, using the Taylor expansion $J(\cdot,f)=1+H_Ef+O(f^2)$, we get 
\begin{align}
&4\gamma\int_{\partial E}
\bigl(v_{E_f}\circ\Phi_f-v_E\bigr)
\varphi(1+H_Ef)\udH \notag
\\
&=
4\gamma\int_{\partial E}
\partial_{\nu_E}v_E\,f\varphi\udH
+
8\gamma\int_{\partial E}\int_{\partial E}
G(x,y)f(y)\varphi(x)\udH_y\udH_x \label{eq:v_Ef-v_E}
\\
&\quad
+
O\big(
\|f\|_{C^0(\partial E)}^\beta
\|f\|_{L^2(\partial E)}
\|\varphi\|_{L^2(\partial E)}
\big). \notag
\end{align}

Let us introduce the notation $\mathcal R_f(\varphi)$ for all error terms satisfying 
\[
|\mathcal R_f(\varphi)|\le C(\delta^*)^\beta\|f\|_{L^2(\bd E)}\|\varphi\|_{L^2(\bd E)}.
\]
Coupling~\eqref{eq:v_Ef-v_E} together with~\cite[equation~(3.13)]{DeKu}, we obtain 
\begin{equation}\label{alex1}
    \begin{split}
       &\int_{\bd E} (H_{E_f}\circ \Phi_f-H_E+4\gamma(v_{E_f} \circ \Phi_f-v_E))\varphi (1+H_E f+R) \\
       &= \int_{\bd E} (\nabla f\cdot \nabla \varphi-|B_E|^2f\varphi) 
       +  8 \gamma \int_{\partial E} \int_{\partial E} G(x,y)f(y)\varphi(x) 
       + 4\gamma \int_{\bd E} \bd_\nu v_Ef\varphi\\
       &\quad+ \int_{\bd E} h\cdot \nabla \varphi +O(| \nabla f|^2)  \varphi + \mathcal R_f(\varphi),
    \end{split}
\end{equation}
where, using the notation of~\cite{DeKu}, $R=O(f^2)+O(|\nabla f|^2)$ and $|h| \le C (|f|+|\nabla f|^2)|\nabla f|$.
Choosing $\varphi=1$ in~\eqref{alex1} yields
\[    
\int_{\bd E} (H_{E_f}\circ \Phi_f-H_E+4\gamma(v_{E_f} \circ \Phi_f-v_E)) (1+H_E f+R) = \mathcal R_f(1)+ \int_{\bd E} O(|f|)+O(|\nabla f|^2) .
\]
Since $E$ is a critical set for ${\mathcal J}$, $H_E+4\gamma v_E$ is constant on $\bd E$, hence
\begin{equation}
    \begin{split}
        &|(\widetilde H_{E_f}-H_E)+4\gamma( \widetilde v_{E_f}- v_E)| \\
        &\le  \left\lvert \fint ( (H_{E_f} \circ \Phi_f- \widetilde H_{E_f})+4\gamma (v_{E_f}\circ \Phi_f-\widetilde v_{E_f}))(H_E f+R)\right\rvert \\
        &\quad+ \left\lvert \fint ( (\widetilde H_{E_f} -H_{E})+4\gamma (\widetilde v_{E_f}-v_E))(H_E f+R)\right\rvert + \mathcal R_f(1)+ \int_{\bd E}O(|f|)+O(|\nabla f|^2)\\
        &\le C\delta^*  \|(H_{E_f}\circ \Phi_f - \widetilde H_{E_f}) +4\gamma (v_{E_f}\circ \Phi_f-\widetilde v_{E_f})  \|_{L^2(\bd E)}\\
        &\quad + C\delta^* |(\widetilde H_{E_f}-H_E)+4\gamma (\widetilde v_{E_f}-v_E)| + C\|f\|_{H^1(\bd E)},
    \end{split}
\end{equation}
where we used H\"{o}lder's inequality, and $C>0$ denotes a constant only depending on $E$. 
In particular,  the previous inequality entails
\begin{equation*}
    |(\widetilde H_{E_f}- H_E)+4\gamma (\widetilde v_{E_f}-v_E)|\le  C\delta^*  \|(H_{E_f}\circ \Phi_f - \widetilde H_{E_f}) +4\gamma (v_{E_f}\circ \Phi_f -\widetilde v_{E_f})  \|_{L^2(\bd E)}+C\|f\|_{H^1(\partial E)}.
\end{equation*}
Testing now~\eqref{alex1} with $\varphi=f$, and using H\"older's and Young's inequalities together with the above estimate, we get for every $\eta>0$
\begin{align*}
     &\int_{\bd E} (|\nabla f|^2-|B_E|^2 f^2) +  8 \gamma \int_{\partial E} \int_{\partial E} G(x,y)f(x)f(y)+ 4\gamma \int_{\bd E} \bd_\nu v_Ef^2 \\
     &=  \int_{\bd E} ((H_{E_f}\circ \Phi_f-H_E)+4\gamma(v_{E_f}\circ \Phi_f-v_E)) f (1+H_E f+R) 
       +C(\delta^*)^\beta \|f\|_{H^1(\partial E)}^2\\
     &\le C\| (H_{E_f}\circ\Phi_f-\widetilde H_{E_f})+4\gamma (v_{E_f} \circ \Phi_f-\widetilde v_{E_f})\|_{L^2(\partial E)}\|f\|_{L^2(\partial E)} + C(\delta^*)^\beta \|f\|_{H^1(\partial E)}^2\\
     & \quad + |(\widetilde H_{E_f}- H_E)+4\gamma (\widetilde v_{E_f}-v_E)|\Big |\int_{\bd E} f (1+H_E f+ R) \Big |\\
     &\le \frac 1\eta C^2\| (H_{E_f}\circ \Phi_f-\widetilde H_{E_f})+4\gamma (v_{E_f}\circ \Phi_f-\widetilde v_{E_f})\|_{L^2(\partial E)}^2 + \eta\|f\|_{H^1(\partial E)}^2 + C(\delta^*)^\beta \|f\|_{H^1(\partial E)}^2,
\end{align*}
where in the last inequality we also used the first assumption in~\eqref{hp alex}. 
The conclusion follows by taking $\delta^*,\eta$ sufficiently small and using Lemma~\ref{lem:coerc}.

\end{proof}

\section{Dynamical Stability}\label{sec:stability}
In this section we employ the stability estimate of Theorem~\ref{teo alex} to prove the dynamical stability of stable sets under the modified Mullins-Sekerka flow. 

 We start by recalling that  a smooth family of sets $(E(t))_{t}\subset\T^N$ with $E(0)=E_0 \subset \T^N$, defined on some (maximal) time interval $(0, T^*)$, is a solution of  the \emph{modified Mullins-Sekerka flow} starting from $E_0$ if it satisfies
\begin{equation}\label{MSintro}
\begin{cases}
V(x,t)= [\bd_{\nu_{E(t)}}w(x,t)] & \text{ $x \in \bd E(t)$, $t \in (0,T^*)$,}\\
\Delta w(x,t)=0 & \text{ $x \in \T^N\setminus \bd E(t)$, $t \in (0,T^*)$,}\\
w(x,t)=H_{E(t)}(x)+4\gamma v(x,t) & \text{ $x \in \bd E(t)$, $t \in (0,T^*)$,}\\
-\Delta v(x,t)=u_{E(t)}(x)-\fint_{\T^N}u_{E(t)}(y) \ud y\,, & \text{ $x \in \T^N$, $t \in (0,T^*)$,}
\end{cases}
\end{equation}
with initial condition $E(0)=E_0$, where both $w$ and $v$ are subject to  periodic boundary conditions and $v$ has zero average. We recall that   $V$ stands for the outer normal velocity of the moving boundary $\bd E$, $u_{E(t)}:=2\chi_{E(t)} -1$ and $[\bd_{\nu_{E(t)}}w(\cdot,t)]$ denotes the jump of the normal derivative of $w(\cdot,t)$ at $\bd E(t)$, i.e.,  $[\bd_{\nu_{E(t)}}w(\cdot,t)]:=\bd_{\nu_{E(t)}}w(\cdot,t)^+-\bd_{\nu_{E(t)}}w(\cdot,t)^-$, with  $w(\cdot,t)^+$ and $w(\cdot,t)^-$ denoting the restrictions of $w(\cdot,t)$ to $\T^N\setminus \overline{E(t)}$ and $E(t)$, respectively.

\subsection{The nonlocal term is a perturbation }
In this section, we combine the results of~\cite{EscNis} and~\cite{XinJiaFah} to obtain existence and Schauder estimates for solutions to~\eqref{MSintro} starting from a sufficiently smooth set. 

Let $E\subset \T^N$ be a fixed smooth reference set, and set
$\Gamma:=\partial E$, $\Omega^-:=E$ and $\Omega^+:=\T^N\setminus \overline E$.
Let $\nu$ denote the outer unit normal to $E$. For a height function
$h:\Gamma\times[0,T]\to \mathbb R$, sufficiently small in $C^1$, we set 
\[
    \Phi_h(p,t):=p+h(p,t)\nu(p),
     \quad \Gamma_h(t)
    :=
    \{p+h(p,t)\nu(p):p\in\Gamma\}.
\]
We denote by $E_h$   the set enclosed by $\Gamma_h$, and write $\Omega_h^-:=E_h$ and $\Omega_h^+:=\T^N\setminus \overline E_h$.
We consider the modified Mullins--Sekerka system
\begin{equation}\label{eq:modified-MS}
\begin{cases}
    V_h=[\partial_{\nu_h}w_h]
        &\text{on }\Gamma_h,\\
    \Delta w_h=0
        &\text{in }\T^N\setminus\Gamma_h,\\
    w_h=H_{\Gamma_h}+4\gamma v_h
        &\text{on }\Gamma_h,\\
    -\Delta v_h=u_{E_h}-m_h
        &\text{in }\T^N,\\
   \int_{\T^N}v_h=0,
\end{cases}
\end{equation}
where $ u_{E_h}:=2\chi_{E_h}-1$ and $m_h:=\fint_{\T^N}u_{E_h}.$

Define $z_h:=w_h-4\gamma v_h $.
Since $\Delta w_h=0$ and  $
    \Delta v_h=-u_{E_h}+m_h$, 
we have $\Delta z_h
    =
    4\gamma(u_{E_h}-m_h)$ in $\T^N\setminus\Gamma_h$
and $ z_h=H_{\Gamma_h}$ on $\Gamma_h$.  
Moreover, since $u_{E_h}-m_h\in L^\infty(\T^N)$, elliptic regularity yields
$v_h \in W^{2,p}(\T^N)$ for every $1\le p <\infty$ and thus  $v_h \in C^{1,\beta}(\T^N)$  for every $\beta\in (0,1)$. In particular, we have $[\partial_{\nu_h}v_h]=0$ on $\Gamma_h$ and $[\partial_{\nu_h}z_h]
    =
    [\partial_{\nu_h}w_h].$ 
Thus~\eqref{eq:modified-MS} is equivalent to
\begin{equation}\label{eq:z-system-moving}
\begin{cases}
    V_h=[\partial_{\nu_h}z_h]
        &\text{on }\Gamma_h,\\
    \Delta z_h=4\gamma(u_{E_h}-m_h)
        &\text{in }\T^N\setminus\Gamma_h,\\
    z_h=H_{\Gamma_h}
        &\text{on }\Gamma_h.
\end{cases}
\end{equation}

Following~\cite[Section 2]{XinJiaFah}, we  apply the Hanzawa transformation~\cite{Han} to   rewrite~\eqref{eq:z-system-moving} on the fixed domains
$\Omega^\pm$. Let $d$ be the signed distance from
$\Gamma$,  and let $\pi$ be the nearest-point projection onto $\Gamma$. Choose $\rho>0$ such that the level sets of $d$ are as regular as $\Gamma$ in the tubular neighborhood $\{|d|<2\rho\}$, and choose $\zeta\in C_c^\infty((-2,2))$ a standard cutoff function with $\zeta\equiv 1$ on $(-1,1)$ and $0\le \zeta \le 1$. We   define the Hanzawa transform $X_h:\T^N \times [0,T]\to\T^N$ by
\begin{equation*}
\begin{aligned}
    X_h(x,t)
    &:=    x+\zeta\!\Big(\frac{d(x)}{\rho}\Big)
    h(\pi(x),t)\nu(\pi(x)) \quad &\text{ if } |d(x)|<2 \rho,\\
    X_h(x,t)
    &:= x \quad &\text{ if } |d(x)|\ge 2 \rho.
\end{aligned}
\end{equation*}
If $\sup_{t\in [0,T]}\|h(\cdot,t)\|_{C^1(\Gamma)}$ is sufficiently small, $X_h(\cdot,t)$ is a diffeomorphism on $\T^N$, uniformly in $t$, and it  satisfies $    X_h(\Gamma,t)=\Gamma_h(t) $ 
and $X_h(\Omega^\pm,t)=\Omega_h^\pm(t).$

Let  
\[
    Y_h:=X_h^{-1},\quad Z_h(y,t):=z_h(X_h(y,t),t).
\]
Since $X_h$ maps the fixed domains $\Omega^\pm$ onto the corresponding moving sets $\Omega_h^\pm$, it holds 
\begin{equation}\label{eq:pullback-characteristic}
    u_{E_h}(X_h(y,t))=u_E(y).
\end{equation} 
Define the transformed elliptic operator (see~\cite{XinJiaFah} for details)
\[
    \mathcal L_h Z_h
    :=
    (\Delta_x(Z_h\circ Y_h))\circ X_h=a_h^{ij}D_{ij}Z_h+b_h^iD_iZ_h,
\]
where $a_h^{ij}
    =
    \bigl(\nabla_x Y_h^i\cdot\nabla_xY_h^j\bigr)\circ X_h$ and $b_h^i
    =
    (\Delta_xY_h^i)\circ X_h.$
Thus, for $h(\cdot,t)\in C^{3,\alpha}(\Gamma)$ in a sufficiently small $C^1$-neighborhood of the origin,
$\mathcal L_h$ is uniformly elliptic and its coefficients depend
smoothly on $h$.
Under the previous change of unknowns,  system~\eqref{eq:z-system-moving} transforms into 
\begin{equation}\label{eq:z-system-fixed}
\begin{cases}
    \mathcal L_h Z_h=4\gamma(u_E-m_h)
        & \text{in }\Omega^\pm,\\
    Z_h=H_{\Gamma_h}\circ \Phi_h
        & \text{on }\Gamma,\\
    V_h\circ \Phi_h=[\partial_{\nu_h}Z_h]\circ \Phi_h
        & \text{on }\Gamma.
\end{cases}
\end{equation}
Note that the right-hand side of the first equation in~\eqref{eq:z-system-fixed} is constant in each of the two domains
\[
    4\gamma(u_E-m_h)
    =
    \begin{cases}
        4\gamma(1-m_h) & \text{in }\Omega^-,\\
        4\gamma(-1-m_h) & \text{in }\Omega^+.
    \end{cases}
\]
Moreover, the normal velocity $V_h$ can be related to the time-derivative of the height function $h$ as follows
\[
    V_h\circ \Phi_h = \bd_t h (\nu\cdot (\nu_h\circ \Phi_h)).
\]
For $\|h\|_{C^1}$ sufficiently small, the term $\nu\cdot (\nu_h\circ \Phi_h)$ in non-zero. Hence,  the third equation in~\eqref{eq:z-system-fixed} can be rewritten as 
\[
    \bd_t h = (\nu\cdot (\nu_h\circ \Phi_h))^{-1} [\partial_{\nu_h}(Z\circ Y_h)]\circ\Phi_h.
\]

After localization via a cut-off function $\eta$ and flattening near a point of $\Gamma$, the operator
$\mathcal L_h$ can be written as a perturbation of the  Laplacian. For a localized unknown $U=\eta Z_h$, one can write
\[
    \Delta U
    =
    \operatorname{div}F_{1,h}+F_{2,h}
\]
in the fixed half-spaces, for some suitable functions $F_{1,h}$ and $F_{2,h}$. 
For more details on the localization and freezing of the coefficients we refer to the proof of~\cite[Lemma 4.2]{XinJiaFah}. 
In particular, $F_{2,h}$ consists of coefficients and cutoff errors together, as well as the additional Ohta--Kawasaki term 
\[
    F_{2,h}^{\mathrm{OK}}
    :=
    4\gamma\eta(u_E-m_h).
\]
To check that this additional term can be treated with the fixed point  argument of~\cite{XinJiaFah}, it is enough to
verify that $F_{2,h}^{OK}$  satisfies   the required estimates  in
the  parabolic H\"older spaces. Specifically, we refer to Lemma 3.4 and Section 5 in~\cite{XinJiaFah}. Since  $m_h=2|E_h|-1$,   
the map $h\mapsto m_h$ is smooth  in a
sufficiently small $C^0$-neighborhood of the origin. Consequently, on
every bounded subset of
$C^{\alpha/3}([0,T];C^0(\Gamma))$, we have
\[
\|m_h\|_{C^{\alpha/3}([0,T])}\leq C
\]
and, for every $h,\tilde h $ in a bounded subset of  $C^{\alpha/3}([0,T];C^0(\Gamma))$, 
\[
\|m_h-m_{\widetilde h}\|_{C^{\alpha/3}([0,T])}
\leq
C\|h-\widetilde h\|_
{C^{\alpha/3}([0,T];C^0(\Gamma))}.
\]
Since $u_E$ is constant in each $\Omega^\pm$ and the cutoff $\eta$ is
smooth, it follows that
\[
\|F^{\mathrm{OK}}_{2,h}\|_
{C^{\alpha,\alpha/3}(\Omega^\pm\times[0,T])}
\leq C\gamma
\]
and
\[
\|F^{\mathrm{OK}}_{2,h}
      -F^{\mathrm{OK}}_{2,\widetilde h}\|_
{C^{\alpha,\alpha/3}(\Omega^\pm\times[0,T])}
\leq
C\gamma
\|h-\widetilde h\|_
{C^{\alpha/3}([0,T];C^0(\Gamma))}.
\]
Hence, the additional term satisfies the boundedness and local Lipschitz
properties required in 
the fixed-point scheme of~\cite[Section 5]{XinJiaFah}.
Moreover, being lower-order, it does not affect the principal
third-order operator, and the fixed-point argument of~\cite{XinJiaFah} applies accordingly,  with constants that may  depend on $\gamma$ but depend on the initial  height function only through a prescribed $C^{3,\alpha}$-bound.
This is the observation made in~\cite{EscNis}, adapted here to the framework of~\cite{XinJiaFah}.

Once short-time existence is shown in $C^{3,\alpha}$, additional regularity follows by the bootstrap argument of~\cite[Section 5]{XinJiaFah}. Note that the term $m_h$ is now a fixed constant $m_h=m_{h_0}$, so  $F^{\mathrm{OK}}_{2,h}$  does not appear in the higher-regularity computations.

\begin{proposition}\label{prop:local-E}
Let $E\subset\T^N$ be a smooth set  and let $\gamma,\alpha\ge0$.
Then, there exists $M_0=M_0(E)>0$ such that for every
$M\in(0,M_0)$ and $\ell\in\N$ there are constants
\[
    T=T(E,\alpha,
    \gamma,
    M)>0,
    \quad
    C=C(E,\alpha,\gamma,
    M)>0, \quad C_\ell=C_\ell(E,\alpha,
    \gamma,
    M)>0
\]
with the following property. 
For every $h_0\in C^{3,\alpha}(\Gamma)$  with $\|h_0\|_{C^{3,\alpha}(\Gamma)}\le M$,  
the modified Mullins--Sekerka flow with initial datum 
$\Gamma_{h_0}$ admits a unique classical solution
\[
    \Gamma(t)=\partial E_{h(\cdot,t)},
    \qquad t\in[0,T],
\]
such that
\begin{align}\label{eq:short-time-basic}
    & h\in
    C([0,T];C^{3,\alpha}(\Gamma))
    \cap
    C^1([0,T];C^\alpha(\Gamma)),\\
    &\sup_{0\le t\le T}
    \|h(\cdot,t)\|_{C^{3,\alpha}(\Gamma)}
    +
    \sup_{0\le t\le T}
    \|\partial_t h(\cdot,t)\|_{C^\alpha(\Gamma)}
    \le C.    \label{eq:short-time-bound}
\end{align}
Moreover, for every integer $\ell\ge1$ it holds 
\begin{equation}\label{eq:short-time-smoothing}
    \sup_{0<t\le T}
    t^{\ell/3}
    \|h(\cdot,t)\|_{C^{3+\ell,\alpha}(\Gamma)}
    \le C_\ell.
\end{equation}
\end{proposition}
We emphasize that in the result above all the constants  are uniform with respect to $h_0$ as long as
$\|h_0\|_{C^{3,\alpha}(\Gamma)}\le M$.

\subsection{Proof of the Main Result}

We start by recalling  a technical lemma whose proof follows  from~\cite[Lemma 3.8]{AFM} combined with~\cite[Theorem 1.1]{AFM}.
\begin{lemma}\label{lem:transl-Ef}
    Let $E\subset \T^N$ be a stable set and $p>N-1$. For every $\varepsilon>0$, there exist constants $C>c>0$, $\rho>0$ such that the following holds. If $F\subset \T^N$ is a $W^{2,p}$ set satisfying 
    \[
    |F|=|E|, \quad \inf_{\sigma\in\T^N}\textnormal{dist}_{W^{2,p}}(E,F+\sigma)\le \rho,
    \]
    then there exist $\sigma\in \T^N$ and $f\in W^{2,p}(\bd E) $ such that $F+\sigma=E_f$,   $ \|f\|_{W^{2,p}(\bd E)}\le C\rho $ and  
    \begin{equation}\label{eq:ortho}
    \Big |\int_{\bd E}f\nu_E \udH\Big | \le \varepsilon \|f\|_{L^2(\partial E)},\quad \Big| \int_{\bd E} f \udH \Big|\le C\|f\|^2_{L^2(\bd E)}.
    \end{equation}
    Moreover, it holds 
    \begin{equation}\label{eq:energy-equivalence}
    c\|f\|_{L^1(\partial E)}^2
    \le
    {\mathcal J}(F)-{\mathcal J}(E)
    \le
    C\|f\|_{H^1(\partial E)}^2.
\end{equation}

\end{lemma}

\begin{remark}\label{rmk:reg-f}
We recall that the function $f$ is defined via a normal projection from the set $F$ and by the implicit function theorem. In particular, following the proof of~\cite[Lemma 3.8]{AFM}, we deduce that, if $F$ is of class $C^k$ for $k\in\N$, then so is  $f$, with bounds depending on the distance $\dist_{C^k}(F,E)$.
\end{remark}

We are now able to show our main result concerning the dynamical stability of the flow. 

\begin{proof}[Proof of Theorem~\ref{thm:main}]
Let $M \in (0,M_0(E))$ with $M_0(E)$ being the constant given by Proposition~\ref{prop:local-E}.  Throughout the proof, the constants $C > c>0$ may depend on
$E$, $\alpha$ and $\gamma$, but not on $E_0$, and may change from line to line.

By assumption there exists $h_0\in C^{3,\alpha}(\bd E)$ such that $E_0=E_{h_0}$ and $\|h_0\|_{C^{3,\alpha}(\bd E)}\le \delta$. For $\delta$ small, depending on $M$, Proposition~\ref{prop:local-E} ensures that the modified Mullins--Sekerka flow $E_t$ starting from $E_0$ exists for a positive time $T>0$ and that there exists $h \in C([0,T];C^{3,\alpha}(\Gamma))
    \cap
    C^1([0,T];C^\alpha(\Gamma))$ satisfying~\eqref{eq:short-time-basic}-\eqref{eq:short-time-smoothing} such that $E_t=E_{h(\cdot,t)}$. Let us also remark that considering smaller $\delta$ does not decrease $T$. \\
\noindent\textbf{Step 1:} We start by proving that, as long as the flow exists, it satisfies  
\[
{\mathcal J}(E_t)-{\mathcal J}(E)\le C e^{-ct}.
\]
We recall the following  identities holding along the smooth flow 
\[
    \dfrac \ud{\ud t}|E_t|=0,\quad \dfrac \ud{\ud t}{\mathcal J}(E_t)= - \int_{\T^N} |\nabla w_t|^2 \ud x,    
\]
where $w_t$ is the harmonic function satisfying $w_t=H_{E_t}+4\gamma v_{E_t}$ on $\partial E_t$.
By Poincare's inequality on $\T^N$ and trace estimates for $w_t$, we deduce 
\begin{equation}\label{smooth comput}
\dfrac  \ud{\ud t}{\mathcal J}(E_t) \le -C\| H_{E_t}+4\gamma v_{E_t} -\overline{H_{E_t}+4\gamma v_{E_t}} \|_{L^2(\bd E_t)}^2.
\end{equation}
Note that the constant $C$, which depends on $E_t$ via the trace  estimates for $w_t$, is uniform in a fixed $C^2$-neighborhood of $E$.

Let $\delta^*>0$ be the constant given by Theorem~\ref{teo alex}, and let $\rho=\rho(\delta^*)>0$ be the constant given by Lemma~\ref{lem:transl-Ef} for $\varepsilon\le \delta^*$. 
By the bound~\eqref{eq:short-time-bound}, we have
\[
\|h(\cdot,t)\|_{C^0(\partial E)} \le \|h_0\|_{C^0(\partial E)} + t \sup_{0\le s  \le T} \|\partial_t h(\cdot,s)\|_{C^0(\partial E)} \le \delta + C t.
\]
Hence, up to taking $T$ and $\delta$ smaller, by interpolation, we  obtain $\|h(\cdot,t)\|_{W^{2,p}(\bd E)}\le \min \{\delta^*, \rho\}$ for every $t \in (0,T)$.
Lemma~\ref{lem:transl-Ef} then  implies that there exist  translations $\sigma_t \in \T^N$ and functions $f(\cdot,t) \in C^1(\bd E)$ such that $E_t+\sigma_t=E_{f(\cdot,t)}$. By ~\eqref{eq:energy-equivalence},~\eqref{smooth comput} and Theorem~\ref{teo alex} we deduce 
 \begin{align*}
    {\mathcal J}(E_{t})-{\mathcal J}(E) &\le C\|f(\cdot,t)\|^2_{H^1(\bd E)} \le C\| H_{E_t}+4\gamma v_{E_t} -\overline{H_{E_t}+4\gamma v_{E_t}} \|^2_{L^2(\bd E_t)}
    \\
    &\le -C \frac{\ud}{\ud t}({\mathcal J}(E_{t})-{\mathcal J}(E)).
\end{align*}
Then Gronwall's inequality implies
\[
 {\mathcal J}(E_t)-{\mathcal J}(E)\le  ( {\mathcal J}(E_0)-{\mathcal J}(E)) e^{-Ct}.
\]
Moreover, using the   bounds of~\eqref{eq:energy-equivalence}, we deduce 
\begin{equation}\label{eq:expo-decay-L1}
    \|f(\cdot,t)\|_{L^1(\bd E)}^2\le   C \delta^2 e^{-Ct}.
\end{equation}

\noindent\textbf{Step 2:}
We adapt the restarting argument  in~\cite[Theorem 0.1, Step 2]{DegDiaKubKub}.
By the uniform short-time existence result of Proposition~\ref{prop:local-E},  the functions $h(\cdot,t)$ are uniformly bounded in $C^k$  for every $t \in [T/2,T)$ and  $k \in \N$. Combining this bound with the
exponential decay~\eqref{eq:expo-decay-L1} and interpolation,
we also deduce that $\|f(\cdot,t)\|_{C^k(\partial E)}\le C_k\delta^{1/2}$  for every $k\in \N$ and $t\in (T/2,T),$ for some $C_k>0$. 

Up to taking $\delta$ smaller, we can thus  restart the flow at time $T/2$ from $E_{f(\cdot,T/2)}$ and extend it up to the time $3T/2$.  By translation invariance and uniqueness of strong solutions,  we have extended the original  flow $E_t$ up to the time $3T/2.$
We can then iterate this procedure to deduce the global existence of the flow.

Applying the same interpolation argument with arbitrary $k$, we obtain  the exponential convergence  in $C^k$ to $E$ of the translated sets $E_t +\sigma_t$.

\noindent\textbf{Step 3:} 
We consider the translations $\sigma_t$ defined by Lemma~\ref{lem:transl-Ef}. By compactness we can find $\tau\in \T^N$ and a sequence $t_n\to+\infty$ such that $\sigma_{t_n}\to \tau$ as $n\to+\infty$. In particular, it holds $E_{t_n} \to E-\tau$ in $C^k$, for every $k \in \N$. For $F,$ $G \subset \T^N$, consider the dissipation 
\[
\mathcal D(F,G):=\int_{F\triangle G} \text{dist}_{\partial G}(x) \ud x = \int_F \sd_G \ud x - \int_G\sd_G \ud x,
\]
where $\sd_G=\text{dist}_{G}-\text{dist}_{G^c}$ denotes the signed distance function to a set $G$. 
Following the computations in Step 3 of the proof of~\cite[Theorem 3.4]{AFJM},
we deduce 
\begin{equation}
    \frac{\ud}{\ud t}\mathcal{D}(E_t,E-\tau) = -\int_{\T^N} \nabla \omega\cdot \nabla w_t \ud x,
\end{equation}
where $\omega$ denotes the harmonic extension of $\text{sd}_{E-\tau}$ to $\T^N\setminus \bd E_t$. Note that, by elliptic estimates, $\|\nabla \omega\|_{L^2(\T^N)}\le C\| \text{sd}_{ E-\tau} \|_{C^1(\bd E_t) }   \le C$. The constant $C$ is uniform along the flow by the uniform $C^2$-bounds on the evolving sets $E_t$ previously obtained.  
By the exponential decay~\eqref{eq:expo-decay-L1} and elliptic estimates  we deduce $\|\nabla w_t\|_{L^2}\le Ce^{-Ct}$, in particular
\[
  \Big |\frac{\ud}{\ud t}\mathcal{D}(E_t,E-\tau)\Big | \le Ce^{-Ct}.
\]
Hence, $\mathcal{D}(E_t,E-\tau)$ admits a limit as $t\to +\infty$. Since $E_{t_n} \to E-\tau$ we deduce that $\mathcal{D}(E_t,E-\tau)\to 0$ as $t\to +\infty$ and that the whole flow $E_t+\tau$ converges  towards $E$.  
This concludes the proof as the exponential convergence follows from Step 2.
\end{proof}

\subsection*{Acknowledgements}
 The authors  thank Vesa Julin and  Massimiliano Morini for helpful discussions and suggestions.
D. De Gennaro was partially funded by the European Union: the European Research
Council (ERC), through StG “ANGEVA”, project number: 101076411. Views and opinions expressed
are however those of the authors only and do not necessarily reflect those of the European Union or
the European Research Council. Neither the European Union nor the granting authority can be held
responsible for them.
D. De Gennaro was   supported by the Italian Ministry of University
and Research (MUR) through the FIS 2 project SiGmA: “Singularities in Geometric
Analysis: Minimal Surfaces and Mean Curvature Flows”, project code FIS-2023-02962
(CUP G53C25000120001).
A. Kubin research has been supported by the Austrian Science Fund (FWF) through grants 10.55776/F65, 10.55776/P35359, 10.55776/Y1292. 
Part of this contribution was completed while D. De Gennaro was visiting A. Kubin at the Technische Universit\"at Wien.

\section*{AI usage statement}
All the proofs contained in this manuscript were written by the authors, who take full responsibility for the correctness of the statements. This article does not contain any mathematical content generated by AI. AI tools were used solely for language editing, in particular to improve the fluency and clarity of the exposition, without contributing to the mathematical content.

\bibliographystyle{abbrv}
\bibliography{bibliography}

 \end{document}